%% file: normsegal.tex
\documentclass{article}
\usepackage{latexsym, amscd, amsfonts, eucal, mathrsfs, amsmath, mathabx, amssymb, amsthm, xypic,xr, stmaryrd, color, enumerate, tikz}
\usepackage{appendix}
\usepackage{multicol}
\usepackage[all]{xy}
\usepackage{hyperref}
\usepackage{fullpage}
\usepackage{tikz-cd}
\usepackage{comment}

\newtheorem{theorem}{Theorem}[section]
\newtheorem{lemma}[theorem]{Lemma}

\newtheorem*{thm*}{Theorem}
\newtheorem*{cor*}{Corollary}

\newtheorem{thm}{Theorem}
\newtheorem{cor}[thm]{Corollary}

\theoremstyle{definition}

\newtheorem{notation}[theorem]{Notation}

\newtheorem{remark}[theorem]{Remark}
\newtheorem{question}[theorem]{Question}
\newtheorem{recollection}[theorem]{Recollection}

\DeclareMathOperator*{\holim}{holim}

\usepackage{cleveref}
\usepackage{scrextend}
\usepackage{mathtools}

\begin{document}

\title{Real topological Hochschild homology and the Segal conjecture}
\author{Jeremy Hahn and Dylan Wilson}
\date{}


\setcounter{tocdepth}{1}
\maketitle
\begin{abstract} 
We give a new proof, independent of Lin's theorem, of the Segal conjecture for the cyclic group of order two.
The key input is a calculation, as a Hopf algebroid, of the Real topological Hochschild homology of $\mathbb{F}_2$.
This determines the $\mathrm{E}_2$-page of the descent spectral sequence for the map $\mathrm{N}\mathbb{F}_2 \to \mathbb{F}_2$, where $\mathrm{N}\mathbb{F}_2$ is the $C_2$-equivariant Hill--Hopkins--Ravenel norm of $\mathbb{F}_2$.  The $\mathrm{E}_2$-page represents a new upper bound on the $RO(C_2)$-graded homotopy of $\mathrm{N}\mathbb{F}_2$, from which the Segal conjecture is an immediate corollary.
\end{abstract}
\tableofcontents


\input{introduction}
\input{thr}

\input{e2term}
\input{ext}

\input{segal}

\input{epilogue}


\bibliographystyle{amsalpha}
\nocite{*}
\bibliography{Bibliography}

\end{document}

%% file: introduction.tex
\section{Introduction}

The Segal conjecture, for the cyclic group $C_2$ of order $2$, is an equivalence
$$\pi^0_{s}(\mathbb{RP}^{\infty}) \cong \widehat{A}(C_2)$$
of the stable cohomotopy of $\mathbb{RP}^{\infty}$ with the completion of the Burnside ring of $C_2$ at the augmentation ideal.  The conjecture follows from the following stronger result of Lin \cite{Lin}:

\begin{theorem}[Lin] \label{LinTate}
Let $\gamma$ denote the canonical line bundle over $\mathbb{RP}^{\infty}$, and for each integer $n>0$ let $\mathbb{RP}^{\infty}_{-n}$ denote the Thom spectrum of $-n \cdot \gamma$.  Then there is an equivalence of spectra
$$\mathbb{RP}^{\infty}_{-\infty} = \mathrm{holim}_{n} \mathbb{RP}^{\infty}_{-n}\simeq (S^{-1})^{\wedge}_2.$$
\end{theorem}

The only known proof of Lin's theorem proceeds via calculation of a certain continuous Ext group
$$\widehat{\mathrm{Ext}}_{\mathcal{A}}\left(\mathrm{H}^*(\mathbb{RP}^{\infty}_{-\infty};\mathbb{F}_2),\mathbb{F}_2\right).$$
The calculation is elegant, and has been generalized through the development of the \emph{Singer construction} \cite{Singer, AGM, nielsen-rognes}.
However, the simplicity of Lin's proof is fundamentally limited by the complexity of the Steenrod algebra $\mathcal{A}$.
The goal of this paper is to provide a new, less computational proof of Lin's theorem.  We cannot avoid calculating a completed Ext group, but the Ext we calculate is over a polynomial coalgebra $\mathbb{F}_2[x]$ rather than the Steenrod algebra $\mathcal{A}$.  We trust the reader will agree that this reduces the complexity of the homological algebra.

\begin{remark}
Just as the Steenrod algebra $\mathcal{A}$ arises as (the dual of) the homology of $\mathbb{F}_2$, the polynomial coalgebra $\mathbb{F}_2[x]$ appears as the \emph{topological Hochschild homology} of $\mathbb{F}_2$.
\end{remark}

To explain our methods, we must review how the Segal conjecture has been both restated and generalized via the language of $C_2$-equivariant stable homotopy theory.

\begin{notation}
In the $C_2$-equivariant stable homotopy category, we use the notation $S=S^0$ to denote the unit object.  This is the $C_2$-equivariant sphere spectrum.  We use $S^{\sigma}$ to denote the $1$-point compactification of the sign representation.  Depending on context, we use $\mathbb{F}_2$ to denote either the field with $2$ elements or the non-equivariant Eilenberg--Maclane spectrum $\mathrm{H}\mathbb{F}_2$.
\end{notation}

\begin{recollection}
In the $C_2$-equivariant stable homotopy category, the morphism 
$$a: S^{-\sigma} \to S^0$$
is adjoint to the inclusion of fixed points into the sign representation $\sigma$.  The \emph{Borel completion} of a $C_2$-spectrum $X$ is the $a$-completion
\[
	X^{\wedge}_a :=
	\holim\left( \cdots\to X/a^n \to X/a^{n-1} \to
	\cdots \to X/a\right).
\]
One says that a $C_2$-spectrum $X$ is \emph{Borel complete} if the natural map $X \to X^{\wedge}_a$ is an equivalence.
\end{recollection}

\begin{theorem}[Lin's theorem, restated] \label{LinEquivariant}
The natural map $S \to S^{\wedge}_a$ is an equivalence after $2$-completion.
\end{theorem}

We will explain the equivalence of the two variants of Lin's theorem in Section \ref{sec:segal}.  In the above form, Lin's theorem has received a substantial generalization.

\begin{recollection}
For any ordinary spectrum $X$, the Hill--Hopkins--Ravenel norm $NX=N_e^{C_2} X$ is a $C_2$-equivariant refinement of the smash product $X \wedge X$, with $C_2$-action given by swapping the two copies of $X$ \cite[\S B.5]{hhr}.
\end{recollection}

A version of the following was first proved in \cite{jones-wegmann} (cf. \cite[Theorem 5.13]{nielsen-rognes}).  As we will recall in Section \ref{sec:segal}, the statement in full generality is a consequence of \cite[III.1.7]{nikolaus-scholze}.
\begin{theorem}[Segal conjecture, strong form] \label{SegalEquivariant}
Let $X$ denote any bounded below spectrum.  Then the natural map
$$\mathrm{N}X \to (\mathrm{N}X)^{\wedge}_a$$
is an equivalence after $2$-completion.
\end{theorem}

Theorem \ref{LinEquivariant} follows from Theorem \ref{SegalEquivariant} by setting $X$ to be the sphere spectrum.  As explained in \cite[III.1.7]{nikolaus-scholze}, Theorem \ref{SegalEquivariant} follows in general from the case $X=\mathbb{F}_2$.  In other words, since $\mathrm{N}\mathbb{F}_2$ is $2$-complete, all statements of Lin's theorem are consequences of the following result:

\begin{theorem} \label{thm:NF2Segal}
The $C_2$-spectrum $\mathrm{N}\mathbb{F}_2$ is Borel complete.
\end{theorem}

Theorem \ref{thm:NF2Segal} is the form in which we will prove the Segal conjecture.
It is important to note that, while Theorem \ref{thm:NF2Segal} tells us that the spectra
$\mathrm{N}\mathbb{F}_2$ and $(\mathrm{N}\mathbb{F}_2)^{\wedge}_a$
coincide, it does not shed light on the homotopy type
of either one.
As we now explain, our main theorem provides a computable upper
bound on the homotopy groups of these spectra, and in this sense our results are stronger than the Segal conjecture.

We prove the following Theorem and Corollary independently of the Segal conjecture:

\begin{thm} \label{thm:main}
Let $\mathbb{F}_2[x]$ be the Hopf algebra
with $x$ primitive of degree $1+\sigma$, and let
$\mathbb{F}_2[a,u]$ be the comodule algebra where the class
$a$ is primitive in degree $-\sigma$, $u$ is in degree $1-\sigma$,
and the coaction is determined
by: 
	\[
	u \mapsto u\otimes 1 + a^2 \otimes x.
	\]  Then there is a spectral sequence
	\[
  E_2 = \widehat{\mathrm{Ext}}^{s,k+\ell\sigma}_{\mathbb{F}_2[x]}(
  \mathbb{F}_2, \mathbb{F}_2[a,u^{\pm 1}])
	\Rightarrow
	\pi_{(k-s)+\ell\sigma}(\mathrm{N}\mathbb{F}_2)^{\wedge}_a.
	\]
Explicitly, the completed $\mathrm{Ext}$ appearing in this $E_2$-page may be calculated as 
\[E_2 = \lim_n\mathrm{Ext}^{s, k+\ell\sigma}_{\mathbb{F}_2[x]/x^{2^n}}
	(\mathbb{F}_2, \mathbb{F}_2[a,u^{\pm 1}]).\]
\end{thm}

\begin{cor} \label{cor:vanishing}
Let $p$ and $q$ denote integers such that $p+q<0$.  Then
	\[
	\pi_{p+q\sigma}(\mathrm{N}\mathbb{F}_2)^{\wedge}_a
	=
	\begin{cases}
	0 & p\ne 0\\
	\mathbb{F}_2\{a^{-q}\} & p=0.
	\end{cases}
	\]
\end{cor}

Corollary \ref{cor:vanishing} follows from straightforward computation of the $\mathrm{Ext}$ groups appearing in Theorem \ref{thm:main}.  As we will explain in Section \ref{sec:segal}, it immediately implies Theorem \ref{thm:NF2Segal} and hence Theorem \ref{SegalEquivariant}.

\begin{remark}
Our proof of Theorem \ref{thm:main} arises by considering the descent spectral sequence for the $C_2$-equivariant morphism
$$\mathrm{N}\mathbb{F}_2 \to \underline{\mathbb{F}_2},$$
where we use $\underline{\mathbb{F}_2}$ to denote the $C_2$-equivariant Eilenberg--Maclane spectrum of the constant Mackey functor $\mathrm{H}\underline{\mathbb{F}_2}$.
The basic descent datum is the $RO(C_2)$-graded homotopy of
$$\underline{\mathbb{F}_2} \otimes_{\mathrm{N}\mathbb{F}_2} \underline{\mathbb{F}_2},$$
which is known as the \emph{Real topological Hochschild homology} of $\underline{\mathbb{F}_2}$.
These $RO(C_2)$-graded homotopy groups were computed as an algebra in \cite{dmps}.
We will need to know them as a \emph{Hopf algebroid}, and not just as an algebra.
Our computation of the Hopf algebroid structure maps is likely of independent interest, and appears in Section \ref{sec:hoch-str}.
\end{remark}

\begin{remark} By the Segal conjecture, the fixed points spectrum
$\mathrm{N}\mathbb{F}_2^{C_2}$ is identified with the
more classical object $(\mathbb{F}_2 \wedge \mathbb{F}_2)^{hC_2}$.
There has been some interest in computing the homotopy groups of these fixed points,
and we give a brief discussion in Section \ref{sec:epilogue}.
\end{remark}

\subsubsection*{Outline}

The spectral sequence in the main theorem is
obtained by taking the $a$-completion of
the relative Adams spectral sequence
for the map $\mathrm{N}\mathbb{F}_2 \to \underline{\mathbb{F}_2}$.
The $E_2$-term of this spectral sequence is governed
by the Hopf algebroid structure on 
$\pi_{\star}(\underline{\mathbb{F}_2} \otimes_{\mathrm{N}\mathbb{F}_2}
\underline{\mathbb{F}_2})$, otherwise known as Real topological
Hochschild homology (cf. \cite{dmps}). We determine
this structure in \S\ref{sec:hoch-str} by comparison
with underlying and geometric fixed points.
In \S\ref{sec:E2-page} we identify the $E_2$-page
for the Borel completion with the indicated limit
of Ext groups. In \S\ref{sec:compute} we compute
these Ext groups, and extract a vanishing result
which implies the Segal conjecture for the group $C_2$.
We give the proof of the Segal conjecture in \S\ref{sec:segal}.
Finally, in \S\ref{sec:epilogue} we indicate a computation
of some low dimensional integer stems, and leave the
reader with a few questions of interest.

\subsubsection*{Conventions} We assume the reader
is acquainted with equivariant homotopy theory at the
level of \cite[\S2,\S3]{hhr}.

\subsubsection*{Acknowledgements}
We thank Danny Shi and Mingcong Zeng for discussions about their works in progress, and for their patience regarding our hasty and error-prone emails.
We thank Hood Chatham for help with some computer calculations we used to explore our $\mathrm{E}_2$-page, as well as for his spectral sequence package.
We thank J.D. Quigley and Tyler Lawson for sharing their unpublished work on the fixed points $(\mathrm{N}\mathbb{F}_2)^{C_2}$, as well as for providing detailed answers to questions.
We thank Mark Behrens for useful discussions, and for sharing a draft of his approach to the Segal conjecture via equivariant homotopy theory.
The first author was supported by the NSF under grant DMS-1803273, and the second author under grant
DMS-1902669.

%% file: thr.tex
\section{The Real Topological Hochschild Homology of \texorpdfstring{\underline{$\mathbb{F}_2$}}{F2}}\label{sec:hoch-str}

Recall from \cite{dmps} that, if $R \in \mathsf{CAlg}^{C_2}$ is
a $C_2$-commutative ring spectrum\footnote{More generally,
this definition makes sense if $R$ is an $\mathbb{E}_{\sigma}$-ring
in the sense of \cite[\S 2.2]{hill}.}, then the \emph{Real topological
Hochschild homology} of $R$ is the $C_2$-spectrum
	\[
	\mathrm{THH}^{\sigma}(R) := R \otimes_{\mathrm{N}R}R.
	\]
The underlying spectrum is the topological Hochschild homology
of $R$ (viewed as an ordinary ring spectrum). The geometric fixed points
are given by
	\[
	(\mathrm{THH}^{\sigma}(R))^{\Phi C_2} \simeq
	R^{\Phi C_2} \otimes_{R} R^{\Phi C_2}.
	\]

Dotto--Moi--Patchkoria--Reeh  computed the 
Real topological Hochschild homology
of $\mathbb{F}_2$ in \cite[Theorem 5.18]{dmps}:

\begin{theorem}
$\mathrm{THH}^{\sigma}(\underline{\mathbb{F}_2})$ is the free 
$\mathbb{E}_1$-$\underline{\mathbb{F}_2}$-algebra on a generator $x$
in degree $\rho$. In particular, there is an isomorphism
of $RO(C_2)$-graded rings:
	\[
	\mathrm{THH}^{\sigma}(\underline{\mathbb{F}_2})_{\star}
	\cong (\underline{\mathbb{F}_{2}})_{\star}[x], \quad |x|=\rho.
	\]
\end{theorem}

If $\mathrm{THH}^{\sigma}(R)_{\star}$ is flat over $R_{\star}$ then the
pair $(R_{\star}, \mathrm{THH}^{\sigma}(R)_{\star})$ forms a Hopf algebroid
in the usual way, since we may identify it with the Hopf algebroid associated
to the relative Adams spectral sequence for the map
	\[
	\mathrm{N}R \to R.
	\]
	
\begin{remark} In the classical setting, the left and right units
for $\mathrm{THH}(R)$ are always homotopic, and, when the
relevant flatness hypothesis is satisfied, the associated Hopf algebroid
is always a Hopf algebra. This is no longer true for real Hochschild homology,
as we will see below. The reason is that the inclusions of the two different
fixed points into $S^{\sigma}$ are not \emph{equivariantly} homotopic.
\end{remark}

Before stating the structure theorem, we recall that the homotopy
groups of $\underline{\mathbb{F}_2}$ are given by
	\[
	\pi_{\star}\underline{\mathbb{F}_2} =
	\mathbb{F}_2[a, u] \oplus \frac{\mathbb{F}_2(a,u)}{\mathbb{F}_2[a,u]}
	\{\theta\}
	\]
where:
	\begin{itemize}
	\item $u: S^{1-\sigma} \to \underline{\mathbb{F}_2}$ is the unique homotopy class
	extending
	the underlying unit $C_{2+} \to \underline{\mathbb{F}_2}$ along the map
	$C_{2+} \to S^{1-\sigma}$.
	\item $\theta: S^{2\sigma} \to S^2$ is the degree 2 cover.
	\end{itemize}

\begin{theorem}\label{thm:hopf-structure} The Hopf algebroid structure
on $((\underline{\mathbb{F}_2})_{\star}, \mathrm{THH}^{\sigma}(\underline{\mathbb{F}_2})_{\star})$
is given as follows:
	\begin{itemize}
	\item The right units on generators are:
		\begin{align*}
		\eta_R(a) &= a,\\
		\eta_R(\theta) &= \theta,\\
		\eta_R(u) &= u + a^2x,\\
		\eta_R(\theta a^{-i}u^{-j}) &= (u+a^2x)^{-j}\frac{\theta}{a^i}.
		\end{align*}
	(Note that the apparently infinite sum in the last formula is
	finite because $\frac{\theta}{a^i}$ is $a$-torsion.)
	\item The comultiplication is determined by
		\[
		\Delta(x) = x\otimes 1 + 1 \otimes x.
		\]
	\end{itemize}
\end{theorem}
\begin{proof} Since $\mathrm{N}S^0 = S^0$ the Hopf algebroid
structure on $\mathrm{THH}^{\sigma}(S^0)$ is trivial. Both $\theta$
and $a$ lie in the Hurewicz image of $S^0 \to \underline{\mathbb{F}_2}$, so
we conclude that they are primitive as indicated.

The element $x$ is primitive for degree reasons: the only
elements in $\pi_{\rho}(\mathrm{THH}^{\sigma}(\underline{\mathbb{F}_2})
\otimes_{\underline{\mathbb{F}_2}} \mathrm{THH}^{\sigma}(\underline{\mathbb{F}_2}))$
are $x \otimes 1$ and $1 \otimes x$.

Now we compute the right unit on $u$. 
Observe that, as a vector space,
$\pi_{1-\sigma}\underline{\mathbb{F}_2}[x] = \mathbb{F}_2\{u, a^2x\}$ so we must have
$\eta_R(u) = \alpha u + \beta a^2x$ for some numbers
$\alpha$ and $\beta$ in $\mathbb{F}_2$. On underlying homotopy we have
$\mathrm{res}(u) = 1$ so that $\alpha = 1$. Now observe that the map
	\[
	\mathrm{N}\mathbb{F}_2 \to \underline{\mathbb{F}_2}
	\]
becomes, upon taking geometric fixed points, the map
	\[
	\mathbb{F}_2 \to \mathbb{F}_2[t]
	\]
where $t = u/a$.

The descent Hopf algebroid for this map has
left and right units $\eta_L, \eta_R: \mathbb{F}_2[t] \to 
\mathbb{F}_2[t] \otimes_{\mathbb{F}_2} \mathbb{F}_2[t]$ given by
$t \mapsto 1\otimes t$ and $t\mapsto t \otimes 1$.
In particular, $\eta_L - \eta_R$ is nonzero on geometric fixed points,
so $\beta$ must be nonzero, completing the proof of the claim.

It remains to compute the right unit on elements of the form
$\theta a^{-i}u^{-j}$. For any $C_2$-spectrum $X$
we have a natural transformation $\partial: \Sigma^{-1}X^{\wedge}_a[a^{-1}] \to X$ arising as the
connecting map in the following arithmetic pullback square:
	\[
	\xymatrix{
	X\ar[r]\ar[d] & X[a^{-1}]\ar[d]\\
	X^{\wedge}_a \ar[r] & X^{\wedge}_a[a^{-1}]
	}
	\]
Moreover, if $X$ is a homotopy
ring, then $\partial$ is a map of $X$-modules.

In particular, we have a commutative diagram:
	\[
	\xymatrix{
	\Sigma^{-1}(\underline{\mathbb{F}_2})^{\wedge}_a[a^{-1}] \ar[r]^-{\eta_R}\ar[d]_{\partial}
	& \Sigma^{-1}(\underline{\mathbb{F}_2} \otimes_{\mathrm{N}\mathbb{F}_2}
	\underline{\mathbb{F}_2})^{\wedge}_a[a^{-1}]\ar[d]^{\partial}\\
	\underline{\mathbb{F}_2} \ar[r]_-{\eta_R} & \underline{\mathbb{F}_2}
	\otimes_{\mathrm{N}\mathbb{F}_2} \underline{\mathbb{F}_2}
	}
	\]
The construction $X \mapsto X^{\wedge}_a[a^{-1}]$
is lax symmetric monoidal, so the map
	\[
	(\underline{\mathbb{F}_2})^{\wedge}_a[a^{-1}] \to \left(\underline{\mathbb{F}_2}
	\otimes_{\mathrm{N}\mathbb{F}_2}\underline{\mathbb{F}_2}\right)^{\wedge}_a[a^{-1}]
	\]
is still a ring map and hence
	\[
	\eta_R(a^{-i-1}u^{-j-1}) = a^{-i-1}\eta_R(u)^{-j-1} 
	= a^{-i-1}(u+a^2x)^{-j-1}.
	\]
On the other hand, $\partial(a^{-i-1}u^{-j-1}) = \theta a^{-i}u^{-j}$
and the result follows.
\end{proof}

%% file: e2term.tex
\section{The Construction of the Spectral Sequence} \label{sec:E2-page}

\begin{theorem} There is a spectral sequence
	\[
	E_2 = \lim_n\mathrm{Ext}^{s, k+\ell\sigma}_{\mathbb{F}_2[x]/x^{2^n}}
	(\mathbb{F}_2[a,u^{\pm 1}]) \Rightarrow
	\pi_{(k-s)+\ell\sigma}(\mathrm{N}\mathbb{F}_2)^{\wedge}_a.
	\]
\end{theorem}

\begin{proof}
Since $\mathrm{N}\mathbb{F}_2$
is connective and 2-complete, there
is an identification\footnote{
The proof is that of 
\cite[Theorem 6.6]{bousfield}, where one
replaces the Postnikov tower with the
$C_2$-equivariant Postnikov tower.}:
	\[
	\mathrm{N}\mathbb{F}_2 
	\simeq \holim_{\Delta} \underline{\mathbb{F}_2}^{\otimes_{\mathrm{N}\mathbb{F}_2}
	\bullet +1}.
	\]
Since $a$-completion preserves homotopy limits, we have
	\[
	(\mathrm{N}\mathbb{F}_2)^{\wedge}_a
	\simeq
	\holim_{\Delta} \left(\underline{\mathbb{F}_2}^{\otimes_{\mathrm{N}\mathbb{F}_2}
	\bullet +1}\right)^{\wedge}_a.
	\]
Thus we get a spectral sequence with $E_1$-term given by
	\[
	E_1^{s, \star+s}=
	\pi_{\star}
	\left(\underline{\mathbb{F}_2}^{\otimes_{\mathrm{N}\mathbb{F}_2}
	s+1}\right)^{\wedge}_a.
	\]
Since $\underline{\mathbb{F}_2} \otimes_{\mathrm{N}\mathbb{F}_2}
\underline{\mathbb{F}_2}$ is free as an $\mathbb{F}_2$-module, the same
is true for each term $\underline{\mathbb{F}_2}^{\otimes_{\mathrm{N}\mathbb{F}_2}s+1}$.
The $a$-completion of $\underline{\mathbb{F}_2}$ has $RO(C_2)$-graded homotopy groups
	\[
	\pi_{\star}(\underline{\mathbb{F}_2})^{\wedge}_a \cong
	\mathbb{F}_2[a, u^{\pm 1}].
	\]
Combined with the computation in the previous
section, we may then identify the $s$th term of the $E_1$-page with
	\[
	\mathbb{F}_2[a, u^{\pm 1}, x_1, ..., x_s]^{\wedge}_{a},
	\]
and the $d_1$-differentials are determined by the coaction
$u \mapsto u + a^2 x$. Here observe that the $a$-completion
is taken in the \emph{graded} sense, 
so the completion of a graded ring
$R$ at the ideal $(a)$
is equivalent to completion with respect to $(a) \cap R_0$.  Thus, we may rewrite
the $s$th term as:
	\begin{align*}
	\mathbb{F}_2[a, u^{\pm 1}, x_1, ..., x_s]^{\wedge}_{a} &\cong
	\mathbb{F}_2[a, u^{\pm 1}, x_1, ..., x_s]^{\wedge}_{(a^2u^{-1}x_1,
	..., a^2u^{-1}x_s)}\\
	&\cong
	\mathbb{F}_2[a, u^{\pm 1}, x_1, ..., x_s]^{\wedge}_{(x_1, ..., x_s)}
	\\&= \lim_n \mathbb{F}_2[a, u^{\pm 1}]\otimes
	\left( \mathbb{F}_2[x]/(x^{2^n})\right)^{\otimes s}.
	\end{align*}
where the completions and limit are understood in the graded setting.
In other words, we may identify the $E_1$-term with the limit
of the cobar complexes:
	\[
	E_1 =
	\lim_n C^*_{\mathbb{F}_2[x]/(x^{2^n})}(\mathbb{F}_2[a, u^{\pm 1}]).
	\]
By the Milnor exact sequence (as in, e.g., \cite[Theorem 3.5.8]{weibel}), this gives the desired computation of the $E_2$-term modulo a possible ${\lim}^1$ contribution.
But for fixed $n$ and tridegree, these groups are finite-dimensional vector
spaces over $\mathbb{F}_2$, so the ${\lim}^1$ vanishes and the result follows.
\end{proof}

%% file: ext.tex
\section{Computation of the \texorpdfstring{$E_2$}{E2}-page} \label{sec:compute}

In this section we compute some information about the
$E_2$-page of the spectral sequence from the previous section.
Our principal aim will be to prove Corollary \ref{cor:vanishing} from the Introduction.

Write $\{E^{(n)}_{r}\}$ for the $x$-adic spectral sequence
	\[
	E^{(n)}_2 = \mathbb{F}_2[a, u, y_0, ..., y_{n-1}]
	\Rightarrow
	\mathrm{Ext}_{\mathbb{F}_2[x]/(x^{2^n})}(\mathbb{F}_2[a, u]),
	\]
where $y_i$ is represented by $[x^{2^i}]$ in the cobar complex,
and write $\{E_r\}$ or $\{E_r^{(\infty)}\}$ for the $x$-adic
spectral sequence
	\[
	E_2 = \mathbb{F}_2[a, u, y_i: i\ge 0] \Rightarrow
	\mathrm{Ext}_{\mathbb{F}_2[x]}(\mathbb{F}_2[a, u]).
	\]

\begin{theorem}\label{thm:effective-ext} We have
ring isomorphisms
	\[
	E_{\infty} =
	\mathbb{F}_2[a, u^{2^{r+1}m}y_r: m, r\ge 0]/(a^{2^{r+1}}u^{2^{r+1}m}y_r).
	\]
	\[
	E^{(n)}_{\infty}
	=
	\mathbb{F}_2[a, u^{2^n},
	u^{2^{r+1}m}y_r: m\ge 0, 0\le r \le n-1]/
	(a^{2^{r+1}}u^{2^{r+1}m}y_r).
	\]
Moreover, there are no nontrivial 
$\mathbb{F}_2[a]$-module extensions.
\end{theorem}

The proof will require the following lemma.

\begin{lemma}\label{lem:coboundary}
The elements $u^{2^{r+1}m}y_r \in E^{(n)}_2(\mathbb{F}_2[a,u])$
are permanent cycles for all $m\ge 0$ and $0\le r \le n-1$.
\end{lemma}
\begin{proof} Let $\delta$ denote the Bockstein
	\[
	\delta: \mathrm{Ext}^0(\mathbb{F}_2[a,u]/(a^{2^{r+1}}))
	\to \mathrm{Ext}^1(\mathbb{F}_2[a, u]).
	\]
In the cobar complex we have
	\[
	d(u^{2^r(2m+1)})
	\equiv u^{2^{r+1}m}a^{2^{r+1}}[x^{2^r}] \mod x^{2^r+1}.
	\]
It follows that $u^{2^{r}(2m+1)}$ is primitive in
$\mathbb{F}_2[a,u]/(a^{2^{r+1}})$ and that
$\delta(u^{2^r(2m+1)})$ is represented
by $u^{2^{r+1}m}[x^{2^r}]$ modulo terms of higher filtration.
This provides a lift of the element $u^{2^{r+1}m}y_r$ to a cocycle
in the cobar complex, which completes the proof.
\end{proof}

\begin{proof}[Proof of Theorem \ref{thm:effective-ext}] 
We will
prove by induction on $t \le n$ that
 $E^{(n)}_{2^{t-1} + 1} = E^{(n)}_{2^{t}}$, and
		\[
		E^{(n)}_{2^{t}} = \mathbb{F}_2[a, u^{2^t}, u^{2^{r+1}m}y_r:
		m\ge 0, 0 \le r \le n-1]/(a^{2^{r+1}}y_r:
		r\le t-1).
		\]
Note that, in this case, $E_{2^n}^{(n)}$ is generated by permanent
cycles so the spectral sequence stops at this page, which is also
the advertised answer.

The base case is trivial, so we assume the result holds for $t$
and turn to the inductive step. Let $I = (i_0, i_1, ..., i_{n-1})$ be a
tuple
of nonnegative integers and denote by $y_I$ the monomial
$y_0^{i_0}y_1^{i_1}\cdots y_{n-1}^{i_{n-1}}$. Given such a monomial,
denote by $m(I)$ the minimal nonzero index in $I$. Then the elements
	\[
	a^mu^{k}y_I
	\]
with
	\begin{itemize}
	\item $m(I) \le t-1$, $m\le 2^{m(I)+1}-1$, and
	$k$ divisible by $2^{m(I)+1}$; or
	\item $m(I) \ge t$ and $k$ divisible by $2^t$
	\end{itemize}
form an $\mathbb{F}_2$-basis for $E^{(n)}_{2^t}$. If
$m(I) \le t-1$
then this element is a product of permanent cycles
by the previous lemma.
Otherwise, using the cobar differential $d(u^{2^t}) =
a^{2^{t+1}}[x^{2^t}]$, we see that
	\[
	d_{2^t}(a^mu^{2^t\ell}y_I)
	=
	a^my_I (\ell (u^{2^t})^{\ell -1} a^{2^{t+1}}y_t),
	\]
and so
	\[
	E^{(n)}_{2^t + 1} = 
	\mathbb{F}_2[a, u^{2^{t+1}}, u^{2^{r+1}m}y_r]/(a^{2^{r+1}}y_r:
	r\le t+1).
	\]
In the cobar complex we have $d(u^{2^{t+1}}) = a^{2^{t+2}}[x^{2^{t+1}}]$,
so $u^{2^{t+1}}$ survives to $E_{2^{t+1}}$ in the $x$-adic spectral sequence,
and the other algebra generators are permanent cycles. This completes
the induction and the theorem follows modulo extension problems.
With notation as in the previous lemma, we note that
$\delta(u^{2^r(2m+1)})$ provides a lift of $u^{2^{r+1}m}y_r$ which is
automatically annihilated by $a^{2^{r+1}}$. 
This resolves the $\mathbb{F}_2[a]$-module extension problem.

The case of $n=\infty$ is essentially the same (or could also be deduced
from the above computation).
\end{proof}

We are now ready to deduce Corollary \ref{cor:vanishing} from the Introduction:

\begin{cor*}
Let $p$ and $q$ denote integers such that $p+q<0$.  Then
	\[
	\pi_{p+q\sigma}(\mathrm{N}\mathbb{F}_2)^{\wedge}_a
	=
	\begin{cases}
	0 & p\ne 0\\
	\mathbb{F}_2\{a^{-q}\} & p=0.
	\end{cases}
	\]
\end{cor*}

\begin{proof}
It suffices to prove this on the $E_2$-term
of the spectral sequence in Theorem \ref{thm:main}.
Indeed, once this result is known on the $E_2$-page,
we see that the classes $a^{-q}$ must be permanent
cycles since their potential targets lie in the vanishing
range.  They cannot be the target of differentials
since they lie in filtration 0.

We will show that, in positive filtration,
the $E_2$-term vanishes when $p+q<0$; the filtration
zero contribution is easily seen to be just $\mathbb{F}_2\{a^{-q}\}$. It further suffices to verify this
vanishing for each group $\mathrm{Ext}_{\mathbb{F}_2[x]/(x^{2^n})}(\mathbb{F}_2[a,u^{\pm1}])$
appearing in the limit defining the $E_2$-page.

Since $u^{2^n}$ is $\mathbb{F}_2[x]/(x^{2^n})$-primitive in 
$\mathbb{F}_2[a, u]$, we have that
	\[
	\mathrm{Ext}_{\mathbb{F}_2[x]/(x^{2^n})}(\mathbb{F}_2[a, u^{\pm 1}])
	=\mathrm{Ext}_{\mathbb{F}_2[x]/(x^{2^n})}
	(\mathbb{F}_2[a, u])
	[(u^{2^n})^{-1}].
	\]
Since $u$ has underlying topological degree $0$, we
can verify the vanishing claim before inverting $u$.
But there it follows immediately from Theorem
\ref{thm:effective-ext}, since each multiplicative generator
of the associated graded, in positive filtration,
satisfies $p+q\ge 0$. 
\end{proof}

%
%
%
%

%% file: segal.tex
\section{The Segal conjecture}\label{sec:segal}

In this section, we prove the Segal conjecture in the following form:

\begin{theorem} \label{thm:segalmain}
Let $X$ denote any bounded below spectrum.  Then the natural map
$$\mathrm{N}X \to (\mathrm{N}X)^{\wedge}_a$$
is an equivalence after $2$-completion.
\end{theorem}

The key point is the following standard observation:

\begin{lemma} \label{lem:tate}
Let $X$ be a bounded below spectrum.  Then, to prove Theorem \ref{thm:segalmain}, it suffices to show that
$$(\mathrm{N}X)[a^{-1}] \to (\mathrm{N}X)^{\wedge}_a[a^{-1}]$$
is an equivalence after $2$-completion.  This in turn is equivalent to the claim that the Tate diagonal
$$X \to (X \wedge X)^{tC_2}$$
is an equivalence after $2$-completion.
\end{lemma}

\begin{proof}
The first part of the lemma follows from the pullback fracture square
	\[
	\xymatrix{
	\mathrm{N}X \ar[r]\ar[d] & (\mathrm{N}X)[a^{-1}] \ar[d]\\
	(\mathrm{N}X)^{\wedge}_a \ar[r] &
	(\mathrm{N}X)^{\wedge}_a[a^{-1}]
	}
	\]
Since this is a pullback, the left hand vertical map is an equivalence after $2$-completion if and only if the right hand vertical map is an equivalence after $2$-completion.
To obtain the second part of the lemma, note that the non-equivariant map underlying $a$ is nullhomotopic.  Thus, the non-equivariant spectra underlying $(\mathrm{N}X)[a^{-1}]$ and $(\mathrm{N}X)^{\wedge}_a[a^{-1}]$ are trivial.  This means that it suffices to check that the map on $C_2$-fixed points
$$\left((\mathrm{N}X)[a^{-1}]\right)^{C_2} \to \left((\mathrm{N}X)^{\wedge}_a[a^{-1}]\right)^{C_2}$$
is an equivalence after $2$-completion.  The above map is identified with the Tate diagonal via 
\cite[\S 2]{behrens-wilson} and 
\cite[\S III.1.5]{nikolaus-scholze}.
\end{proof}

Next, we refer to an argument of Nikolaus--Scholze \cite[proof of III.1.7]{nikolaus-scholze}, where they show that the Tate diagonal
$$X \to (X \wedge X)^{tC_2}$$
is an equivalence after $2$-completion, for all bounded below $X$, if it is when $X=\mathbb{F}_2$.  Given this, we are ready to prove Theorem \ref{thm:segalmain}.

\begin{proof}[Proof of Theorem \ref{thm:segalmain}]
We are reduced to proving that 
$$(\mathrm{N}\mathbb{F}_2)[a^{-1}] \to (\mathrm{N}\mathbb{F}_2)^{\wedge}_a[a^{-1}]$$
is an equivalence. 
Observe that, when $a$ acts invertibly on
a $C_2$-spectrum $Y$, we have
	\[
	\cdot a^q: \pi_{p+q\sigma}Y
	\stackrel{\cong}{\longrightarrow}
	\pi_p(Y^{C_2}).
	\]
Since $(\mathrm{N}\mathbb{F}_2[a^{-1}])^{C_2}
=\mathbb{F}_2$, we deduce that
$\pi_{\star}\mathrm{N}\mathbb{F}_2[a^{-1}]
=\mathbb{F}_2[a^{\pm 1}]$.
So it suffices to show that:
	\[
	\pi_{\star}(\mathrm{N}\mathbb{F}_2)^{\wedge}_a[a^{-1}]
	=\mathbb{F}_2[a^{\pm 1}]. 
	\]
Multiplication by $a$ decreases underlying 
topological dimension, so it further suffices to show that,
when $p+q<0$, we have:
	\[
	\pi_{p+q\sigma}
	(\mathrm{N}\mathbb{F}_2)^{\wedge}_a
	=
	\begin{cases}
	0 & p\ne 0\\
	\mathbb{F}_2\{a^{-q}\} & p=0
	\end{cases}
	\]
This is the statement of Corollary \ref{cor:vanishing}.
\end{proof}

\begin{remark}
Setting $X=S^0$, it follows from the above that
$$(S^0)^{\wedge}_{2} \simeq (S^0 \wedge S^0)^{tC_2} \simeq (S^0)^{tC_2}.$$
After identifying $(S^0)^{tC_2}$ with $\Sigma \mathbb{RP}^{\infty}_{-\infty}$, the original version of Lin's theorem follows.
\end{remark}

%% file: epilogue.tex
\section{Epilogue}\label{sec:epilogue}

\subsection*{Integer stems}
Over the last few years, there have been several attempts to understand the homotopy groups of the non-equivariant spectrum
$$\left(\mathrm{N}\mathbb{F}_2 \right)^{C_2} = (\mathbb{F}_2 \wedge \mathbb{F}_2)^{hC_2}.$$
This seems especially interesting in light of forthcoming work of Mingcong Zeng and Lennart Meier, which uses the equivalence
$$\Phi^{C_2}N_{C_2}^{C_4}\mathrm{BP}\mathbf{R} \simeq \mathrm{N}\mathbb{F}_2$$
to relate these homotopy groups to the slice spectral sequence differentials studied by Hill-Shi-Wang-Xu \cite{c4slice}

The most straightforward approach to these homotopy groups is via the homotopy fixed point spectral sequence.  However, even the $E_2$-page, given by the group cohomology $H^*(C_2; \mathcal{A}_*)$, is largely unknown at this time \cite{crossley-whitehouse}.  Another approach, pursued independently in unpublished work by J.D. Quigley and Tyler Lawson, is to use the non-equivariant $\mathbb{F}_2$-Adams spectral sequence.
Quigley was able to use the Adams spectral sequence to obtain some results about $\pi_*(\mathrm{N}\mathbb{F}_2^{C_2})$ for $* < 10$.  We suspect that the use of the equivariant $\mathbb{F}_2$-Adams spectral sequence for $\pi_{\star}\mathrm{N}\mathbb{F}_2$ would lead to similar complications as those encountered by Lawson and Quigley.

The \emph{relative} Adams spectral sequence of this paper, restricted to integer stems $\pi_{p+0\sigma} \mathrm{N}\mathbb{F}_2$, provides yet another route to these homotopy groups.  We draw the $E_2$-page below, with each circle representing a single copy of $\mathbb{F}_2$:

\begin{remark}
The $E_2$-page is easy to compute
if one is only interested in integer stems.
In this case, even on the $E_1$-page,
the only possible contributions come from
$\mathrm{Ext}_{\mathbb{F}_2[x]}(\mathbb{F}_2[a,u])$
so
we may forego inverting $u$ and the corresponding completion.
\end{remark}

\begin{figure}[ht]
  \centering
	$E_2$ page of the spectral sequence in degrees contributing to $\pi_{p+0\sigma} \mathrm{N}\mathbb{F}_2$ \\
   \includegraphics[trim={3.9cm 17cm 9cm 4cm}, clip, scale=1]{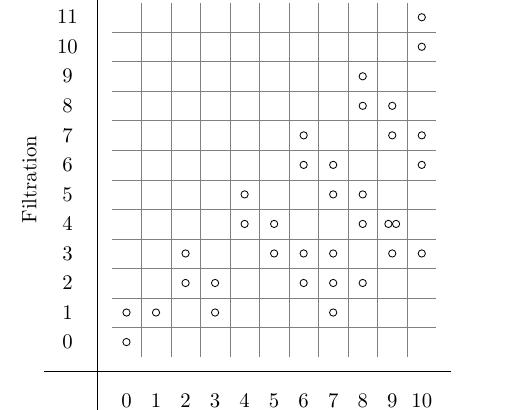}
\end{figure}

Using a low-dimensional cell structure
for $\mathrm{N}\mathbb{F}_2$, one can show that
there is a nontrivial extension between the classes in degrees $(0,0)$ and $(0,1)$, and in particular that $\pi_0 \mathrm{N}\mathbb{F}_2 \cong \mathbb{Z}/4\mathbb{Z}$.
One can also prove that the class in $(1,1)$ detects 
$\eta$.

We believe, but have not verified, that there is a $d_2$ differential from the class in degree $(5,3)$ to the class in degree $(4,5)$.  This should be
a consequence of
a whole family of $d_2$ differentials
	\[
	d_2(y_{i+1}) = (ay_0)y_i^2,
	\]
connected to each other via power operations.  We suspect that an equivariant analog of work of 
Kahn \cite{kahn}, as generalized by Bruner 
\cite[\S VI]{bmms}, could establish this family of differentials.

\subsection*{Further Questions}

\begin{question}
Can the spectral sequence in this paper be used to recover any of the exotic differentials
established by Hill-Shi-Wang-Xu in \cite{c4slice}?  While the $a_{\lambda}$-inverted slice spectral sequence also converges to the $RO(C_2)$-graded homotopy of $\mathrm{N}\mathbb{F}_2$, the two spectral sequences differ greatly on the $E_2$-page.
It is conceivable that they become much more similar after running the
slice differentials from Hill-Hopkins-Ravenel, since these implement
the $a$-torsion visible on our $E_2$-page.
\end{question}

\begin{question}
Can our method of proof be generalized to deduce the Segal conjecture for elementary $p$-groups?  The Segal conjecture for elementary $p$-groups is the key computational input for the Segal conjecture in general
\cite{AGM}.  For the group $C_p$, it seems that a study of $\underline{\mathbb{F}_p} \otimes_{N^{C_p}_{e} \mathbb{F}_p} \underline{\mathbb{F}_p}$ would be relevant.
\end{question}

\begin{question}
The groups $\mathrm{Ext}_{\mathbb{F}_2[x]}(\mathbb{F}_2[a,u])$
are much smaller than the version with $u$ inverted.  Real motivic homotopy theory provides a setting similar to $C_2$-equivariant homotopy theory in which the negative cone is not present (see, e.g., \cite{behrens-shah}).
Is there a notion of motivic topological Hochschild homology of $\mathbb{F}_2$
whose homotopy groups are the
Hopf algebroid $\mathbb{F}_2[a, u, x]$?
\end{question}